\documentclass{article}%
\usepackage{amsmath}
\usepackage{amsfonts}
\usepackage{amssymb}
\usepackage{graphicx}%
\providecommand{\U}[1]{\protect\rule{.1in}{.1in}}
\newtheorem{theorem}{Theorem}

\newtheorem{corollary}[theorem]{Corollary}

\newtheorem{definition}[theorem]{Definition}

\newenvironment{proof}[1][Proof]{\noindent\textbf{#1.} }{\ \rule{0.5em}{0.5em}}
\begin{document}

\title{A new class of generalized Genocchi polynomials}
\author{N. I. Mahmudov\\Eastern Mediterranean University\\Gazimagusa, TRNC, Mersiin 10, Turkey \\Email: nazim.mahmudov@emu.edu.tr}
\maketitle

\begin{abstract}
The main purpose of this paper is to introduce and investigate a new class of
generalized Genocchi polynomials based on the $q$-integers. The $q$-analogues
of well-known formulas are derived. The $q$-analogue of the
Srivastava--Pint\'{e}r addition theorem is obtained.

\end{abstract}

\section{ Introduction}

Throughout this paper, we always make use of the following notation:
$\mathbb{N}$ denotes the set of natural numbers, $\mathbb{N}_{0}$ denotes the
set of nonnegative integers, $\mathbb{R}$ denotes the set of real numbers,
$\mathbb{C}$ denotes the set of complex numbers.

The $q$-shifted factorial is defined by
\[
\left(  a;q\right)  _{0}=1,\ \ \ \left(  a;q\right)  _{n}=%
{\displaystyle\prod\limits_{j=0}^{n-1}}
\left(  1-q^{j}a\right)  ,\ \ \ n\in\mathbb{N},\ \ \ \left(  a;q\right)
_{\infty}=%
{\displaystyle\prod\limits_{j=0}^{\infty}}
\left(  1-q^{j}a\right)  ,\ \ \ \ \left\vert q\right\vert <1,\ \ a\in
\mathbb{C}.
\]
The $q$-numbers and $q$-numbers factorial is defined by%
\[
\left[  a\right]  _{q}=\frac{1-q^{a}}{1-q}\ \ \ \left(  q\neq1\right)
;\ \ \ \left[  0\right]  _{q}!=1;\ \ \ \ \left[  n\right]  _{q}!=\left[
1\right]  _{q}\left[  2\right]  _{q}...\left[  n\right]  _{q}\ \ \ \ \ n\in
\mathbb{N},\ \ a\in\mathbb{C}%
\]
respectively. The $q$-polynomail coefficient is defined by%
\[
\left[
\begin{array}
[c]{c}%
n\\
k
\end{array}
\right]  _{q}=\frac{\left(  q;q\right)  _{n}}{\left(  q;q\right)
_{n-k}\left(  q;q\right)  _{k}}.
\]
The $q$-analogue of the function $\left(  x+y\right)  ^{n}$ is defined by%
\[
\left(  x+y\right)  _{q}^{n}:=%
{\displaystyle\sum\limits_{k=0}^{n}}
\left[
\begin{array}
[c]{c}%
n\\
k
\end{array}
\right]  _{q}q^{\frac{1}{2}k\left(  k-1\right)  }x^{n-k}y^{k},\ \ \ n\in
\mathbb{N}_{0}.
\]
In the standard approach to the $q$-calculus two exponential function are
used:%
\begin{align*}
e_{q}\left(  z\right)   &  =\sum_{n=0}^{\infty}\frac{z^{n}}{\left[  n\right]
_{q}!}=\prod_{k=0}^{\infty}\frac{1}{\left(  1-\left(  1-q\right)
q^{k}z\right)  },\ \ \ 0<\left\vert q\right\vert <1,\ \left\vert z\right\vert
<\frac{1}{\left\vert 1-q\right\vert },\ \ \ \ \ \ \ \\
E_{q}\left(  z\right)   &  =\sum_{n=0}^{\infty}\frac{q^{\frac{1}{2}n\left(
n-1\right)  }z^{n}}{\left[  n\right]  _{q}!}=\prod_{k=0}^{\infty}\left(
1+\left(  1-q\right)  q^{k}z\right)  ,\ \ \ \ \ \ \ 0<\left\vert q\right\vert
<1,\ z\in\mathbb{C}.\
\end{align*}
From this form we easily see that $e_{q}\left(  z\right)  E_{q}\left(
-z\right)  =1$. Moreover,%
\[
D_{q}e_{q}\left(  z\right)  =e_{q}\left(  z\right)  ,\ \ \ \ D_{q}E_{q}\left(
z\right)  =E_{q}\left(  qz\right)  ,
\]
where $D_{q}$ is defined by%
\[
D_{q}f\left(  z\right)  :=\frac{f\left(  qz\right)  -f\left(  z\right)
}{qz-z}.
\]

Carlitz has introduced the $q$-Bernoulli numbers and polynomials in
\cite{carlitz}. Srivastava and Pinter proved some relations and theorems
between the Bernoulli polynomials and Euler polynomials in \cite{sri1}. They
also gave some generalizations of these polynomials. In \cite{kim1}%
-\cite{kim9}, Kim et al. investigated some properties of the $q$-Euler
polynomials and Genocchi polynomials. They gave some recurrence relations. In
\cite{cenkci}, Cenkci et al. gave the $q$-extension of Genocchi numbers in a
different manner. In \cite{kim5}, Kim gave a new concept for the $q$-Genocchi
numbers and polynomials. In \cite{simsek}, Simsek et al. investigated the
$q$-Genocchi zeta function and $l$-function by using generating functions and
Mellin transformation.

\begin{definition}
The $q$-Bernoulli numbers $\mathfrak{B}_{n,q}^{\left(  \alpha\right)  }$ and
polynomials $\mathfrak{B}_{n,q}^{\left(  \alpha\right)  }\left(  x,y\right)  $
in $x,y$ of order $\alpha$ are defined by means of the generating function
functions:%
\begin{align*}
\left(  \frac{t}{e_{q}\left(  t\right)  -1}\right)  ^{\alpha}  &  =\sum
_{n=0}^{\infty}\mathfrak{B}_{n,q}^{\left(  \alpha\right)  }\frac{t^{n}%
}{\left[  n\right]  _{q}!},\ \ \ \left\vert t\right\vert <2\pi,\\
\left(  \frac{t}{e_{q}\left(  t\right)  -1}\right)  ^{\alpha}e_{q}\left(
tx\right)  E_{q}\left(  ty\right)   &  =\sum_{n=0}^{\infty}\mathfrak{B}%
_{n,q}^{\left(  \alpha\right)  }\left(  x,y\right)  \frac{t^{n}}{\left[
n\right]  _{q}!},\ \ \ \left\vert t\right\vert <2\pi.
\end{align*}

\end{definition}

\begin{definition}
The $q$-Genocchi numbers $\mathfrak{G}_{n,q}^{\left(  \alpha\right)  }$ and
polynomials $\mathfrak{G}_{n,q}^{\left(  \alpha\right)  }\left(  x,y\right)  $
in $x,y$ are defined by means of the generating functions:%
\begin{align*}
\left(  \frac{2t}{e_{q}\left(  t\right)  +1}\right)  ^{\alpha}  &  =\sum
_{n=0}^{\infty}\mathfrak{G}_{n,q}^{\left(  \alpha\right)  }\frac{t^{n}%
}{\left[  n\right]  _{q}!},\ \ \ \left\vert t\right\vert <\pi,\\
\left(  \frac{2t}{e_{q}\left(  t\right)  +1}\right)  ^{\alpha}e_{q}\left(
tx\right)  E_{q}\left(  ty\right)   &  =\sum_{n=0}^{\infty}\mathfrak{G}%
_{n,q}^{\left(  \alpha\right)  }\left(  x,y\right)  \frac{t^{n}}{\left[
n\right]  _{q}!},\ \ \ \left\vert t\right\vert <\pi.
\end{align*}

\end{definition}

It is obvious that%
\begin{align*}
\mathfrak{B}_{n,q}^{\left(  \alpha\right)  }  &  =\mathfrak{B}_{n,q}^{\left(
\alpha\right)  }\left(  0,0\right)  ,\ \ \ \lim_{q\rightarrow1^{-}%
}\mathfrak{B}_{n,q}^{\left(  \alpha\right)  }\left(  x,y\right)
=B_{n}^{\left(  \alpha\right)  }\left(  x+y\right)  ,\ \ \ \lim_{q\rightarrow
1^{-}}\mathfrak{B}_{n,q}^{\left(  \alpha\right)  }=B_{n}^{\left(
\alpha\right)  },\\
\mathfrak{G}_{n,q}^{\left(  \alpha\right)  }  &  =\mathfrak{G}_{n,q}^{\left(
\alpha\right)  }\left(  0,0\right)  ,\ \ \ \lim_{q\rightarrow1^{-}%
}\mathfrak{G}_{n,q}^{\left(  \alpha\right)  }\left(  x,y\right)
=G_{n}^{\left(  \alpha\right)  }\left(  x+y\right)  ,\ \ \ \lim_{q\rightarrow
1^{-}}\mathfrak{G}_{n,q}^{\left(  \alpha\right)  }=G_{n}^{\left(
\alpha\right)  }.
\end{align*}
Here $B_{n}^{\left(  \alpha\right)  }\left(  x\right)  $ and $E_{n}^{\left(
\alpha\right)  }\left(  x\right)  $ denote the classical Bernoulli and
Genocchi polynomials of order $\alpha$ are defined by%
\[
\left(  \frac{t}{e^{t}-1}\right)  ^{\alpha}e^{tx}=\sum_{n=0}^{\infty}%
B_{n}^{\left(  \alpha\right)  }\left(  x\right)  \frac{t^{n}}{\left[
n\right]  _{q}!}\ \ \ \ \ \text{and\ \ \ \ }\left(  \frac{2}{e^{t}+1}\right)
^{\alpha}e^{tx}=\sum_{n=0}^{\infty}G_{n}^{\left(  \alpha\right)  }\left(
x\right)  \frac{t^{n}}{\left[  n\right]  _{q}!}.
\]

The aim of the present paper is to obtain some results for the $q$-Genocchi
polynomials. The $q$-analogues of well-known results, for example, Srivastava
and Pint\'{e}r \cite{pinter}, Cheon \cite{cheon}, etc., can be derived from
these $q$-identities. The formulas involving the $q$-Stirling numbers of the
second kind, $q$-Bernoulli polynomials and $q$-Bernstein polynomials are also
given. Furthermore some special cases are also considered.

The following elementary properties of the $q$-Genocchi polynomials
$\mathfrak{E}_{n,q}^{\left(  \alpha\right)  }\left(  x,y\right)  $ of order
$\alpha$ are readily derived from Definition. We choose to omit the details involved.

\textbf{Property 1. }\emph{Special values of the }$q$\emph{-Genocchi
polynomials of order }$\alpha$\emph{:}%
\[
\mathfrak{E}_{n,q}^{\left(  0\right)  }\left(  x,0\right)  =x^{n}%
,\ \ \ \mathfrak{E}_{n,q}^{\left(  0\right)  }\left(  0,y\right)  =q^{\frac
{1}{2}n\left(  n-1\right)  }y^{n}.
\]

\textbf{Property 2.}\emph{ Summation formulas for the }$q$\emph{-Genocchi
polynomials of order }$\alpha$\emph{:}%
\begin{align*}
\mathfrak{E}_{n,q}^{\left(  \alpha\right)  }\left(  x,y\right)   &  =%
{\displaystyle\sum\limits_{k=0}^{n}}
\left[
\begin{array}
[c]{c}%
n\\
k
\end{array}
\right]  _{q}\mathfrak{E}_{k,q}^{\left(  \alpha\right)  }\left(  x+y\right)
_{q}^{n-k},\ \ \ \mathfrak{E}_{n,q}^{\left(  \alpha\right)  }\left(
x,y\right)  =%
{\displaystyle\sum\limits_{k=0}^{n}}
\left[
\begin{array}
[c]{c}%
n\\
k
\end{array}
\right]  _{q}\mathfrak{E}_{n-k,q}^{\left(  \alpha-1\right)  }\mathfrak{E}%
_{k,q}\left(  x,y\right)  ,\\
\mathfrak{G}_{n,q}^{\left(  \alpha\right)  }\left(  x,y\right)   &
=\sum_{k=0}^{n}\left[
\begin{array}
[c]{c}%
n\\
k
\end{array}
\right]  _{q}q^{\left(  n-k\right)  \left(  n-k-1\right)  /2}\mathfrak{G}%
_{k,q}^{\left(  \alpha\right)  }\left(  x,0\right)  y^{n-k}=\sum_{k=0}%
^{n}\left[
\begin{array}
[c]{c}%
n\\
k
\end{array}
\right]  _{q}\mathfrak{G}_{k,q}^{\left(  \alpha\right)  }\left(  0,y\right)
x^{n-k},\\
\mathfrak{G}_{n,q}^{\left(  \alpha\right)  }\left(  x,0\right)   &
=\sum_{k=0}^{n}\left[
\begin{array}
[c]{c}%
n\\
k
\end{array}
\right]  _{q}\mathfrak{G}_{k,q}^{\left(  \alpha\right)  }x^{n-k}%
,\ \ \ \mathfrak{G}_{n,q}^{\left(  \alpha\right)  }\left(  0,y\right)
=\sum_{k=0}^{n}\left[
\begin{array}
[c]{c}%
n\\
k
\end{array}
\right]  _{q}q^{\left(  n-k\right)  \left(  n-k-1\right)  /2}\mathfrak{G}%
_{k,q}^{\left(  \alpha\right)  }y^{n-k}.
\end{align*}

\textbf{Property 3.}\emph{ Difference equations:}%
\begin{align*}
\mathfrak{G}_{n,q}^{\left(  \alpha\right)  }\left(  1,y\right)  +\mathfrak{G}%
_{n,q}^{\left(  \alpha\right)  }\left(  0,y\right)   &  =2\left[  n\right]
_{q}\mathfrak{G}_{n-1,q}^{\left(  \alpha-1\right)  }\left(  0,y\right)  ,\\
\mathfrak{G}_{n,q}^{\left(  \alpha\right)  }\left(  x,0\right)  +\mathfrak{G}%
_{n,q}^{\left(  \alpha\right)  }\left(  x,-1\right)   &  =2\left[  n\right]
_{q}\mathfrak{G}_{n-1,q}^{\left(  \alpha-1\right)  }\left(  x,-1\right)  .
\end{align*}

\textbf{Property 4. }\emph{Differential relations:}%
\[
D_{q,x}\mathfrak{G}_{n,q}^{\left(  \alpha\right)  }\left(  x,y\right)
=\left[  n\right]  _{q}\mathfrak{G}_{n-1,q}^{\left(  \alpha\right)  }\left(
x,y\right)  ,\ \ \ D_{q,y}\mathfrak{G}_{n,q}^{\left(  \alpha\right)  }\left(
x,y\right)  =\left[  n\right]  _{q}\ \mathfrak{G}_{n-1,q}^{\left(
\alpha\right)  }\left(  x,qy\right)  .
\]

\textbf{Property 5. }\emph{Addition theorem of the argument:}%
\[
\mathfrak{E}_{n,q}^{\left(  \alpha+\beta\right)  }\left(  x,y\right)  =%
{\displaystyle\sum\limits_{k=0}^{n}}
\left[
\begin{array}
[c]{c}%
n\\
k
\end{array}
\right]  _{q}\mathfrak{E}_{n-k,q}^{\left(  \alpha\right)  }\left(  x,0\right)
\mathfrak{E}_{k,q}^{\left(  \beta\right)  }\left(  0,y\right)  .
\]

\textbf{Property 6. }\emph{Recurrence Relationships:}%
\[
\mathfrak{G}_{n,q}^{\left(  \alpha\right)  }\left(  \frac{1}{m},y\right)
+\sum_{k=0}^{n}\left[
\begin{array}
[c]{c}%
n\\
k
\end{array}
\right]  _{q}\left(  \frac{1}{m}-1\right)  _{q}^{n-k}\mathfrak{G}%
_{k,q}^{\left(  \alpha\right)  }\left(  0,y\right)  =2\left[  n\right]
_{q}\sum_{k=0}^{n-1}\left[
\begin{array}
[c]{c}%
n-1\\
k
\end{array}
\right]  _{q}\left(  \frac{1}{m}-1\right)  _{q}^{n-1-k}\mathfrak{G}%
_{k,q}^{\left(  \alpha-1\right)  }\left(  0,y\right)  .
\]

\section{Explicit relationship between the $q$-Genocchi and the $q$-Bernoulli
polynomials}

In this section we prove an interesting relationship between the $q$-Genocchi
polynomials $\mathfrak{G}_{n,q}^{\left(  \alpha\right)  }\left(  x,y\right)  $
of order $\alpha$ and the $q$-Bernoulli polynomials. Here some $q$-analogues
of known results will be given. We also obtain new formulas and their some
special cases below.

\begin{theorem}
\label{S-P1}For $n\in\mathbb{N}_{0}$, the following relationship%
\begin{align*}
\mathfrak{G}_{n,q}^{\left(  \alpha\right)  }\left(  x,y\right)   &  =%
{\displaystyle\sum\limits_{k=0}^{n}}
\frac{1}{m^{n-k-1}\left[  k+1\right]  _{q}}\left[  2\left[  k+1\right]  _{q}%
{\displaystyle\sum\limits_{j=0}^{k}}
\left[
\begin{array}
[c]{c}%
k\\
j
\end{array}
\right]  _{q}\frac{1}{m^{k-j}}\mathfrak{G}_{j,q}^{\left(  \alpha-1\right)
}\left(  x,-1\right)  \right. \\
&  -\left.
{\displaystyle\sum\limits_{j=0}^{k+1}}
\left[
\begin{array}
[c]{c}%
k+1\\
j
\end{array}
\right]  _{q}\frac{1}{m^{k+1-j}}\mathfrak{G}_{j,q}^{\left(  \alpha\right)
}\left(  x,-1\right)  -\mathfrak{G}_{k+1,q}^{\left(  \alpha\right)  }\left(
x,0\right)  \right]  \mathfrak{B}_{n-k,q}\left(  0,my\right)  .
\end{align*}
holds true between the $q$-Genocchi and the $q$-Bernoulli polynomials..
\end{theorem}

\begin{proof}
Using the following identity%
\[
\left(  \frac{2t}{e_{q}\left(  t\right)  +1}\right)  ^{\alpha}e_{q}\left(
tx\right)  E_{q}\left(  ty\right)  =\left(  \frac{2t}{e_{q}\left(  t\right)
+1}\right)  ^{\alpha}e_{q}\left(  tx\right)  \cdot\frac{e_{q}\left(  \frac
{t}{m}\right)  -1}{t}\cdot\frac{t}{e_{q}\left(  \frac{t}{m}\right)  -1}\cdot
E_{q}\left(  \frac{t}{m}my\right)
\]
we have%
\begin{align*}%
{\displaystyle\sum\limits_{n=0}^{\infty}}
\mathfrak{G}_{n,q}^{\left(  \alpha\right)  }\left(  x,y\right)  \frac{t^{n}%
}{\left[  n\right]  _{q}!}  &  =\frac{m}{t}%
{\displaystyle\sum\limits_{n=0}^{\infty}}
\left(
{\displaystyle\sum\limits_{k=0}^{n}}
\left[
\begin{array}
[c]{c}%
n\\
k
\end{array}
\right]  _{q}\frac{1}{m^{n-k}}\mathfrak{G}_{k,q}^{\left(  \alpha\right)
}\left(  x,0\right)  -\mathfrak{G}_{n,q}^{\left(  \alpha\right)  }\left(
x,0\right)  \right)  \frac{t^{n}}{\left[  n\right]  _{q}!}%
{\displaystyle\sum\limits_{n=0}^{\infty}}
\mathfrak{B}_{n,q}\left(  0,my\right)  \frac{t^{n}}{m^{n}\left[  n\right]
_{q}!}\\
&  =%
{\displaystyle\sum\limits_{n=1}^{\infty}}
\left(
{\displaystyle\sum\limits_{k=0}^{n}}
\left[
\begin{array}
[c]{c}%
n\\
k
\end{array}
\right]  _{q}\frac{1}{m^{n-1-k}}\mathfrak{G}_{k,q}^{\left(  \alpha\right)
}\left(  x,0\right)  -m\mathfrak{G}_{n,q}^{\left(  \alpha\right)  }\left(
x,0\right)  \right)  \frac{t^{n-1}}{\left[  n\right]  _{q}!}%
{\displaystyle\sum\limits_{n=0}^{\infty}}
\mathfrak{B}_{n,q}\left(  0,my\right)  \frac{t^{n}}{m^{n}\left[  n\right]
_{q}!}\\
&  =%
{\displaystyle\sum\limits_{n=0}^{\infty}}
\left(
{\displaystyle\sum\limits_{k=0}^{n+1}}
\left[
\begin{array}
[c]{c}%
n+1\\
k
\end{array}
\right]  _{q}m^{k}\mathfrak{G}_{k,q}^{\left(  \alpha\right)  }\left(
x,0\right)  -m^{n+1}\mathfrak{G}_{n+1,q}^{\left(  \alpha\right)  }\left(
x,0\right)  \right)  \frac{t^{n}}{m^{n}\left[  n+1\right]  _{q}!}%
{\displaystyle\sum\limits_{n=0}^{\infty}}
\mathfrak{B}_{n,q}\left(  0,my\right)  \frac{t^{n}}{m^{n}\left[  n\right]
_{q}!}\\
&  =%
{\displaystyle\sum\limits_{n=0}^{\infty}}
{\displaystyle\sum\limits_{k=0}^{n}}
\frac{1}{m^{n}\left[  k+1\right]  _{q}}\left(
{\displaystyle\sum\limits_{j=0}^{k+1}}
\left[
\begin{array}
[c]{c}%
k+1\\
j
\end{array}
\right]  _{q}m^{j}\mathfrak{G}_{j,q}^{\left(  \alpha\right)  }\left(
x,0\right)  -m^{k+1}\mathfrak{G}_{k+1,q}^{\left(  \alpha\right)  }\left(
x,0\right)  \right)  \mathfrak{B}_{n-k,q}\left(  0,my\right)  \frac{t^{n}%
}{\left[  n\right]  _{q}!}.
\end{align*}
It remains to use Poperty 6.
\end{proof}

Since $\mathfrak{G}_{n,q}^{\left(  \alpha\right)  }\left(  x,y\right)  $ is
not symmetric with respect to $x$ and $y$ we can prove a different from of the
above theorem. It should be stressed out that Theorems \ref{S-P1} and
\ref{S-P11} coincide in the limiting case when $q\rightarrow1^{-}.$

\begin{theorem}
\label{S-P11}For $n\in\mathbb{N}_{0}$, the following relationship%
\begin{align*}
\mathfrak{G}_{n,q}^{\left(  \alpha\right)  }\left(  x,y\right)   &  =%
{\displaystyle\sum\limits_{k=0}^{n}}
\left[
\begin{array}
[c]{c}%
n\\
k
\end{array}
\right]  _{q}\frac{1}{m^{n-k-1}\left[  k+1\right]  _{q}}\left[  2\left[
k+1\right]  _{q}\sum_{j=0}^{k}\left[
\begin{array}
[c]{c}%
k\\
j
\end{array}
\right]  _{q}\left(  \frac{1}{m}-1\right)  _{q}^{k-j}\mathfrak{G}%
_{j,q}^{\left(  \alpha-1\right)  }\left(  0,y\right)  \right. \\
&  -\left.  \sum_{j=0}^{k+1}\left[
\begin{array}
[c]{c}%
k+1\\
j
\end{array}
\right]  _{q}\left(  \frac{1}{m}-1\right)  _{q}^{k+1-j}\mathfrak{G}%
_{j,q}^{\left(  \alpha\right)  }\left(  0,y\right)  -\mathfrak{G}%
_{k+1,q}\left(  0,y\right)  \right]  \mathfrak{B}_{n-k,q}\left(  mx,0\right)
\end{align*}
holds true between the $q$-Genocchi and the $q$-Bernoulli polynomials.
\end{theorem}

\begin{proof}
The proof is based on the following identity%
\[
\left(  \frac{2t}{e_{q}\left(  t\right)  +1}\right)  ^{\alpha}e_{q}\left(
tx\right)  E_{q}\left(  ty\right)  =\left(  \frac{2t}{e_{q}\left(  t\right)
+1}\right)  ^{\alpha}E_{q}\left(  ty\right)  \cdot\frac{e_{q}\left(  \frac
{t}{m}\right)  -1}{t}\cdot\frac{t}{e_{q}\left(  \frac{t}{m}\right)  -1}\cdot
e_{q}\left(  \frac{t}{m}mx\right)  .
\]

\end{proof}

Next we discuss some special cases of Theorems \ref{S-P1} and \ref{S-P11}. By
noting that%
\[
\mathfrak{G}_{j,q}^{\left(  0\right)  }\left(  0,y\right)  =q^{\frac{1}%
{2}j\left(  j-1\right)  }y^{j},\ \ \ \ \ \mathfrak{G}_{j,q}^{\left(  0\right)
}\left(  x,-1\right)  =\left(  x-1\right)  _{q}^{j}%
\]
we deduce from Theorems \ref{S-P1} and \ref{S-P11} Corollary \ref{C:1} below.

\begin{corollary}
\label{C:1}For $n\in\mathbb{N}_{0}$, $m\in\mathbb{N}$ the following
relationship%
\begin{align*}
\mathfrak{G}_{n,q}\left(  x,y\right)   &  =%
{\displaystyle\sum\limits_{k=0}^{n}}
\left[
\begin{array}
[c]{c}%
n\\
k
\end{array}
\right]  _{q}\frac{1}{m^{n-k-1}\left[  k+1\right]  _{q}}\left[  2\left[
k+1\right]  _{q}\sum_{j=0}^{k}\left[
\begin{array}
[c]{c}%
k\\
j
\end{array}
\right]  _{q}\left(  \frac{1}{m}-1\right)  _{q}^{k-j}q^{\frac{1}{2}j\left(
j-1\right)  }y^{j}\right. \\
&  -\left.  \sum_{j=0}^{k+1}\left[
\begin{array}
[c]{c}%
k+1\\
j
\end{array}
\right]  _{q}\left(  \frac{1}{m}-1\right)  _{q}^{k+1-j}\mathfrak{G}%
_{j,q}\left(  0,y\right)  -\mathfrak{G}_{k+1,q}\left(  0,y\right)  \right]
\mathfrak{B}_{n-k,q}\left(  mx,0\right)  ,
\end{align*}%
\begin{align*}
\mathfrak{G}_{n,q}\left(  x,y\right)   &  =%
{\displaystyle\sum\limits_{k=0}^{n}}
\left[
\begin{array}
[c]{c}%
n\\
k
\end{array}
\right]  _{q}\frac{1}{m^{n-k-1}\left[  k+1\right]  _{q}}\left[  2\left[
k+1\right]  _{q}%
{\displaystyle\sum\limits_{j=0}^{k}}
\left[
\begin{array}
[c]{c}%
k\\
j
\end{array}
\right]  _{q}\frac{1}{m^{k-j}}\left(  x-1\right)  _{q}^{j}\right. \\
&  -\left.
{\displaystyle\sum\limits_{j=0}^{k+1}}
\left[
\begin{array}
[c]{c}%
k+1\\
j
\end{array}
\right]  _{q}\frac{1}{m^{k+1-j}}\mathfrak{G}_{j,q}\left(  x,-1\right)
-\mathfrak{G}_{k+1,q}\left(  x,0\right)  \right]  \mathfrak{B}_{n-k,q}\left(
0,my\right)  .
\end{align*}
holds true between the $q$-Bernoulli polynomials and $q$-Euler polynomials.
\end{corollary}

\begin{corollary}
\label{C:2} For $n\in\mathbb{N}_{0}$, $m\in\mathbb{N}$ the following
relationship holds true.%
\begin{align}
G_{n}\left(  x+y\right)   &  =%
{\displaystyle\sum\limits_{k=0}^{n}}
\left(
\begin{array}
[c]{c}%
n\\
k
\end{array}
\right)  \frac{2}{k+1}\left(  \left(  k+1\right)  y^{k}-G_{k+1,q}\left(
y\right)  \right)  B_{n-k}\left(  x\right)  ,\label{g1}\\
G_{n}\left(  x+y\right)   &  =\sum_{k=0}^{n}\left(
\begin{array}
[c]{c}%
n\\
k
\end{array}
\right)  \frac{1}{m^{n-k-1}\left(  k+1\right)  }\left[  2\left(  k+1\right)
G_{k}\left(  y+\frac{1}{m}-1\right)  -G_{k+1}\left(  y+\frac{1}{m}-1\right)
-G_{k+1}\left(  y\right)  \right]  B_{n-k,q}\left(  mx\right)  \label{g2}%
\end{align}
between the classical Genocchi polynomials and the classical Bernoulli polynomials.
\end{corollary}

Note that the formula (\ref{g2}) is new for the classical polynomials.

In terms of the $q$-Genocchi numbers $\mathfrak{G}_{k,q}^{\left(
\alpha\right)  }$, by setting $y=0$ in Theorem \ref{S-P1}, we obtain the
following explicit relationship between the $q$-Genocchi polynomials
$\mathfrak{G}_{k,q}^{\left(  \alpha\right)  }$ of order $\alpha$ and the
$q$-Bernoulli polynomials.

\begin{corollary}
The following relationship holds true:%
\begin{align*}
\mathfrak{G}_{n,q}^{\left(  \alpha\right)  }\left(  x,0\right)   &  =%
{\displaystyle\sum\limits_{k=0}^{n}}
\left[
\begin{array}
[c]{c}%
n\\
k
\end{array}
\right]  _{q}\frac{1}{m^{n-k-1}\left[  k+1\right]  _{q}}\left[  2\left[
k+1\right]  _{q}\sum_{j=0}^{k}\left[
\begin{array}
[c]{c}%
k\\
j
\end{array}
\right]  _{q}\left(  \frac{1}{m}-1\right)  _{q}^{k-j}\mathfrak{G}%
_{j,q}^{\left(  \alpha-1\right)  }\right. \\
&  -\left.  \sum_{j=0}^{k+1}\left[
\begin{array}
[c]{c}%
k+1\\
j
\end{array}
\right]  _{q}\left(  \frac{1}{m}-1\right)  _{q}^{k+1-j}\mathfrak{G}%
_{j,q}^{\left(  \alpha\right)  }-\mathfrak{G}_{k+1,q}^{\left(  \alpha\right)
}\right]  \mathfrak{B}_{n-k,q}\left(  mx,0\right)  .
\end{align*}

\end{corollary}

\begin{corollary}
\label{C:3}For $n\in\mathbb{N}_{0}$ the following relationship holds true.%
\[
\mathfrak{G}_{n,q}\left(  x,y\right)  =%
{\displaystyle\sum\limits_{k=0}^{n}}
\left[
\begin{array}
[c]{c}%
n\\
k
\end{array}
\right]  _{q}\frac{2}{\left[  k+1\right]  _{q}}\left[  \left[  k+1\right]
_{q}q^{\frac{1}{2}k\left(  k-1\right)  }y^{k}-\mathfrak{G}_{k+1,q}\left(
0,y\right)  \right]  \mathfrak{B}_{n-k,q}\left(  x,0\right)  .
\]

\end{corollary}

\begin{corollary}
\label{C:4}For $n\in\mathbb{N}_{0}$ the following relationship holds true.%
\begin{align*}
\mathfrak{G}_{n,q}\left(  x,0\right)   &  =-%
{\displaystyle\sum\limits_{k=0}^{n}}
\left[
\begin{array}
[c]{c}%
n\\
k
\end{array}
\right]  _{q}\frac{2}{\left[  k+1\right]  _{q}}\mathfrak{G}_{k+1,q}%
\mathfrak{B}_{n-k,q}\left(  x,0\right)  ,\\
\mathfrak{G}_{n,q}  &  =-%
{\displaystyle\sum\limits_{k=0}^{n}}
\left[
\begin{array}
[c]{c}%
n\\
k
\end{array}
\right]  _{q}\frac{2}{\left[  k+1\right]  _{q}}\mathfrak{G}_{k+1,q}%
\mathfrak{B}_{n-k,q}.
\end{align*}

\end{corollary}

\end{document}